\documentclass[10pt,a4paper]{amsart}
\usepackage[utf8x]{inputenc}
\usepackage{amsmath, amsthm}
\usepackage{amsfonts}
\usepackage{amssymb}
\usepackage{graphicx}
\usepackage{parskip}
\usepackage{makecell}
\usepackage{amsfonts,amsthm,latexsym,amsmath,amssymb,amscd,amsmath}

\usepackage[all]{xy}
\usepackage{color}

\usepackage{epstopdf}

\newtheorem{theorem}{Theorem}
\newtheorem{conj}{Conjecture}
\newtheorem{lemma}{Lemma}
\newtheorem{proposition}[theorem]{Proposition}
\theoremstyle{definition}
\newtheorem{remark}{Remark}

\newtheorem{question}{Question}

\newcommand{\lm}{\ell m}
\newcommand{\newsection}{\bigskip}
\newcommand{\newsubsection}{\medskip}

\newcommand{\lideal}{(}
\newcommand{\rideal}{)}

\DeclareMathOperator{\HS}{HS}

\title[On ideals gen. by two generic quadratic forms in the 
exterior alg.]{On ideals generated by two generic quadratic forms in the exterior algebra}

\author[V. Crispin]{Veronica Crispin Qui{\~n}onez}
\address[V. Crispin]{Department of Mathematics, Uppsala University, S-751 06, Uppsala, Sweden} 
\email{veronica.crispin@math.uu.se}

\author[S. Lundqvist]{Samuel Lundqvist}
\author[G. Nenashev]{ Gleb Nenashev}
\address[S. Lundqvist, G. Nenashev]{   Department of Mathematics,
            Stockholm University,
            S-106 91, Stockholm, Sweden}
\email{samuel@math.su.se, nenashev@math.su.se}

\begin{document}

\maketitle

\begin{abstract}
Based on the structure theory of pairs of skew-symmetric matrices, 
we give a conjecture for the Hilbert series of the exterior algebra modulo the ideal generated by two generic quadratic forms. We show that the conjectured series is an upper bound in the coefficient-wise sense, and we determine a 
majority of the coefficients. We also conjecture that the series is equal to the series of the squarefree polynomial ring modulo the ideal generated by the squares of  two generic linear forms.
\end{abstract}

\section{Introduction} \label{sec:1}

Let $S_n = \mathbb{C}[x_1,\ldots,x_n]$ and let $I = \lideal f_1,\ldots,f_r\rideal$ be a homogeneous ideal in $S_n$. There is a longstanding conjecture due to Fr\"oberg \cite{frobergconj} about the minimal Hilbert series of $S_n/I$, namely that it equals

$$\left[ \frac{\prod (1-t^{d_r})}{(1-t)^n} \right],$$ where $d_r$ denotes the degree of the form $f_r$, and where $[1+a_1 t + a_2 t^2 + \cdots ]$ means truncate before the first non-positive term.
It is clear that the minimal series is attained when the $f_i$'s are generic forms, that is, forms with mutually algebraically independent coefficients over $\mathbb{Q}$. 
The conjecture is proved only in some special cases, including  $n=2$ by Fr\"oberg~\cite{frobergconj}, $n=3$ by Anick~\cite{anick} and $r\le n+1$ by Stanley~\cite{stanley}.

Let $E_n$ denote the exterior algebra on $n$ generators over $\mathbb{C}$.
 It is natural to believe that the Hilbert series of $E_n/\lideal f_1,\ldots,f_r\rideal$, with the $f_i'$s being generic forms of even degree, should be equal to $\left[ {(1+t)^n} {\prod (1-t^{d_r})} \right]$.
 Moreno-Soc\'ias and Snellman \cite{morenosnellman} showed that this is true when $I$ is a principal ideal generated by an even element.  
 However, Fr\"oberg and L\"ofwall  \cite[Theorem 10.1]{froberg} showed that for two generic quadratic forms in $E_5$, the Hilbert series equals $1 + 5t + 8t^2 + t^3$, while $[(1+t)^5(1-t^2)^2] = 1 + 5t + 8t^2$.

Starting in 2016, a group connected to the \emph{Stockholm problem solving seminar} has been working on the problem of determining the minimal Hilbert series for quotients of exterior algebras. 
In this paper, the case of two quadratic forms is considered. The results on principal ideals generated by an element of odd degree is presented in a separate paper \cite{lu-ni}.

The outline is as follows.
In Section \ref{sec:can}, we will use the structure theory of skew-symmetric matrices in order to reduce the problem of finding the minimal series among all forms $f$ and $g$ to certain pairs $\lideal \bar{f},\bar{g}\rideal$ of \emph{canonical forms}. 

In Section \ref{sec:comb}
 we combine the results from the previous section with combinatorial methods 
and we show that the result by Fr\"oberg and L\"ofwall is just the tip of the iceberg. We are able to determinine the first $\lfloor n/3 \rfloor + 1$ coefficients, and the last non-zero coefficient of the generic series, and we have a conjecture for the series as a whole.

\begin{conj}\label{conj:paths}
Let $f$ and $g$ be generic quadratic forms in the exterior algebra $E_n$. Then the Hilbert series of $E_n/\lideal f,g\rideal$ is equal to 
$1 + a(n,1) t + a(n,2) t^2 + \cdots +  \cdots + a(n,\lceil\frac{n}{2}\rceil)t^{\lceil\frac{n}{2}\rceil}$, where $a(n,s)$ is the number of lattice paths inside the rectangle $(n+2-2s)\times (n+2)$ from the bottom left corner to the top right corner with moves of two types: $(x,y)\rightarrow (x+1,y+1) \textrm{\ or\ } (x-1,y+1)$.
\end{conj}


We show in Theorem \ref{thm:upper} that the conjectured series is an \emph{upper bound} for $E_n/\lideal f,g \rideal$. This gives us a situation that is opposite to the Fr\"oberg conjecture in the commutative setting, in which case a lower bound for the generic series is known.

In Section \ref{sec:squarefree}, we turn to the Hilbert series of $S_n/\lideal x_1^2,\ldots,x_n^2,\ell_1^2,\ell_2^2 \rideal$, $\ell_1$ and $\ell_2$ being generic linear forms in $S_n$.
It is known that the Hilbert series of $S_n/\lideal x_1^2,\ldots,x_n^2,\ell_1^2\rideal$ is equal to $[(1+t)^n(1-t^2)]$ (see~\cite{stanley}). (In fact, this is a special case of the result in loc. cit. that every artinian monomial complete intersection has the Strong Lefschetz property.)
But the Hilbert series of $S_n/\lideal x_1^2,\ldots,x_n^2,\ell_1^2,\ell_2^2\rideal$ is not 
equal to $[(1+t)^n(1-t^2)^2]$. The first published counterexamples were given in~\cite{froberghollman}. In the last part of the paper, we present a surprising connection between the Hilbert series of 
$S_n/\lideal x_1^2,\ldots,x_n^2,\ell_1^2,\ell_2^2\rideal$ and $E_n/( f,g)$, where $f$ and $g$ are generic quadratic forms.

\begin{conj}\label{conj:paths-sqfree}
Let $\ell_1$ and $\ell_2$ be generic linear forms in $S_n$. Then the Hilbert series of $S_n/\lideal x_1^2,\ldots,x_n^2,\ell_1^2,\ell_2^2\rideal$ is equal to 
$1 + a(n,1) t + a(n,2) t^2 + \cdots$, where $a(n,s)$ is as in Conjecture~\ref{conj:paths}.
\end{conj}

Although the conjectures give the same Hilbert series for different and non-isomorphic algebras,  Conjecture~\ref{conj:paths-sqfree} is weaker than Conjecture~\ref{conj:paths}. 
\begin{theorem}\label{thm:conj12} The following are equivalent
\begin{itemize}
\item Conjecture~\ref{conj:paths} holds for any even $n$;
\item Conjecture~\ref{conj:paths-sqfree} holds for any $n$.
\end{itemize}
\end{theorem}

We prove Theorem~\ref{thm:conj12} in Section \,4.2.

\newsection

\section{Canonical forms in the exterior algebra} \label{sec:can}

Here we recall the structure theory for pairs of skew-symmetric matrices over the complex field and apply it on quadratic forms in the exterior algebra  $E_n$.

Let $A,B\in \mathbb{C}^{n\times n}$ be a pair of quadratic matrices. A matrix pencil is the sum $A+\lambda B$ with a parameter  $\lambda\in\mathbb{C}$. Two matrix pencils $A+\lambda B$ and $\tilde A+\lambda\tilde B$ are said to be strictly $\mathbb{C}$-congruent if there is a nonsingular  $P\in\mathbb {C}^{n\times n}$ such that $P^t(A+\lambda B)P=\tilde A+\lambda\tilde B$ for all $\lambda$. 

A block diagonal matrix with the diagonal blocks $M_1, \ldots, M_r$ will be called a direct sum of $M_i$'s and denoted by $M_1\oplus\ldots\oplus M_r$.

\newsubsection

\subsection{Canonical forms of complex matrix pencils}

 In \cite{T} the author presents canonical matrices of complex skew-symmetric matrix pencils under congruence transformation. The description is based on a result for matrices over an algebraically closed field of characteristic $\ne 2$.   
Below we present the three canonical forms derived from section 5 in \cite{T}, in a slightly different setting and with another notation, obtained from section 5 in \cite{lancaster}, provided with comments about rank.\\

\begin{theorem} \label{thm:Ctypes} Let $A$ and $B$ be complex skew-symmetric quadratic matrices. Then the matrix pencil $A+\lambda B$ is strictly $\mathbb C$-congruent to a direct sum of possibly a zero matrix followed by diagonal blocks of three canonical types, where each type may occur some finite number of times (possibly zero) and may be of different sizes.

\begin{multline*} M_{I, 2k+1}=A_{I}+\lambda B_I=\\
\left( \begin{array}{rrrrrrr}
 &  &  &0  & \ldots & 1 &  0\\
 & O_{k} &  &\vdots &\reflectbox{$\ddots$} &   \vdots  & \vdots\\
 &  &  &  1& \ldots & 0 & 0\\
0 &\ldots  &-1  &  &  &  & 0\\
\vdots & \reflectbox{$\ddots$} & \vdots &  &  O_k &   & \vdots\\
-1 &  \ldots & 0 &  &  & & 0\\
0 & \ldots & 0 & 0 & \ldots & 0 & 0 
\end{array} \right)  + \lambda
\left( \begin{array}{rrrrrrr}
 &  &  &0  & 0 & \ldots &  1\\
 & O_{k} &  &\vdots & \vdots &   \reflectbox{$\ddots$}  & \vdots\\
 &  &  &  0& 1 & \ldots & 0\\
0 &\ldots  &0  & 0 & 0 & \ldots & 0\\
0 & \ldots & -1 & 0 &  &  & \\
\vdots &  \reflectbox{$\ddots$} & \vdots & \vdots &  & O_k & \\
-1 & \ldots & 0 & 0 &  &  & 
\end{array} \right) 
\end{multline*}
A diagonal block of type $I$ is a sum of rank $2k$ matrices. 
\begin{multline*} M_{II, 2l}=A_{II}+\lambda B_{II}=\\
\left( \begin{array}{rrrrrr}
 &  &  & 0 & \ldots & 1\\
 & O_{l}  & & \vdots & \reflectbox{$\ddots$} & \vdots\\
 &  &  & 1 &  \ldots & 0\\
 0 & \ldots & -1 &  &  & \\
 \vdots& \reflectbox{$\ddots$} & \vdots  &  & O_l  & \\
  -1&  \ldots & 0 & & & \\
\end{array} \right) + \lambda
\left( \begin{array}{rrrrrrrr}
 & &  &  & 0 & 0 & \ldots& 0\\
 & &  & & 0 & 0 & \ldots & 1\\
 & & O_{l} &  & \vdots & \vdots &  \reflectbox{$\ddots$} &\vdots \\
 & &  & & 0 & 1 & \ldots & 0\\
0 & 0 & \ldots & 0 &  & & &\\
0 & 0& \ldots  & -1 & & & & \\
\vdots  & \vdots  &  \reflectbox{$\ddots$} & \vdots & & & O_l &\\
0 & -1 & \ldots & 0& & & &\\
\end{array} \right) 
 \end{multline*}
This type is a sum of a rank $2l$ respectively rank $2l-2$ matrix.
\begin{multline*}M_{III, 2m}=A_{III}+\lambda B_{III}=\\ 
\left( \begin{array}{rrrrrrrr}
 & &  &  & 0 & 0 & \ldots& \alpha\\
 & &  & & 0 & \vdots & \reflectbox{$\ddots$}  & 1\\
 & & O_{m} &  & \vdots & \reflectbox{$\ddots$}  &  \reflectbox{$\ddots$} &\vdots \\
 & &  & & \alpha & 1 & \ldots & 0\\
0 & 0 & \ldots & -\alpha &  & & &\\
0 & \vdots& \reflectbox{$\ddots$}   & -1 & & & & \\
\vdots  & \reflectbox{$\ddots$}  &  \reflectbox{$\ddots$} & \vdots & & & O_m &\\
-\alpha & -1 & \ldots & 0& & & &\\
\end{array} \right)  + \lambda
\left( \begin{array}{rrrrrr}
 &  &  & 0 & \ldots & 1\\
 & O_{m}  & & \vdots & \reflectbox{$\ddots$} & \vdots\\
 &  &  & 1 &  \ldots & 0\\
 0 & \ldots & -1 &  &  & \\
 \vdots& \reflectbox{$\ddots$} & \vdots  &  & O_m  & \\
  -1&  \ldots & 0 & & & \\
\end{array} \right)
\end{multline*}
{\rm{[Here $\alpha$  is a finite  root (not necessarily non-zero) of an elementary divisor pair $(\alpha+\lambda)^m, (\alpha+\lambda)^m$ of $A+\lambda B$.]}}\\
A diagonal block of type $III$ is a sum of a matrix of rank $2m$ for nonzero $\alpha$ ($2m-2$ for $\alpha =0$)  and a matrix of rank $2m$.
\end{theorem}

\newsubsection

\subsection{Quadratic forms over the complex numbers}

Any quadratic form $\tilde f=\sum_{1\leq i<j\leq n} a_{ij}x_ix_j\in E_n$ can be written as the product ${\bf x}^t\tilde A{\bf x}$, where ${\bf x}=(x_i)_{i=1}^n\in{\mathbb{C}}^n$ is a (vertical) vector and $\tilde A=(\frac{a_{ij}}{2})\in \mathbb{ C}^{n\times n}$ is a skew-symmetric matrix, that is, $a_{ji}=-a_{ij}$. If $P^t\tilde AP= A$ for some nonsingular matrix $P\in \mathbb{C}^{n\times n}$, then by change of variables ${\bf x}=P{\bf y}$, we have ${\bf x}^t\tilde A{\bf x}={\bf y}^t A{\bf y}$.

\begin{theorem} \label{thm:Cstructure}
Let $\tilde f,\tilde g\in E_{n}$ be two generic quadratic forms. Then, after a proper change of variables, we can write them as: 
\begin{enumerate}
\item $f= \sum_{i=1}^{k} x_ix_{k+1+i}\; and\; g= \sum_{i=1}^{k} x_ix_{k+i},\; if\; n=2k+1;$

\item  $f= \sum_{i=1}^k \alpha_ix_{2i-1}x_{2i}\; and\;  g= \sum_{i=1}^k x_{2i-1}x_{2i}\; with\; nonzero\; \alpha_i\; such\; that\; \alpha_i\ne \alpha_j\; for\; i\ne j,\; if\; n=2k.$
\end{enumerate}
\end{theorem}

\begin{proof}
Let $\tilde A$ and $\tilde B$ be the matrices that correspond to $\tilde f$ and $\tilde g$. By  Theorem~\ref{thm:Ctypes} and the discussion above we can write $\langle \tilde f,\tilde g\rangle = ( {\bf y}^t A{\bf y},{\bf y}^t B{\bf y})$, where $A+\lambda B = P^t(\tilde A+\lambda\tilde B)P$ is a direct sum of canonical diagonal blocks. The ranks of $\tilde A$ and $A$ are the same, as well as the ranks of $\tilde B$ and $B$.  We know that a matrix corresponding to a generic quadratic form has maximal rank. Hence, we will examine which diagonal sums of canonical pairs $A,B$ have the same property, with the additional condition to maximaze the rank of $A+\lambda B$ as well.



The odd and even case differ due to the fact that the rank of a skew-symmetric matrix is always even, so that for odd $n$ the matrix pencil $A+\lambda B$ is always singular. If $n$ is even, the case to consider is when the matrix pencil is strictly regular, that is, when both $A$ and $B$ are nonsingular.

{ \ \ ({\it i}).} Let $n=2k+1$. We see that for all $\lambda$ the matrix pencil $M_{I, 2k+1}$ has maximal rank $2k$ as well as each of the included terms $A_I$ and $B_I$. Hence, the skew-symmetric matrices $\tilde A, \tilde B$, corresponding to generic quadratic forms, having maximal rank are strictly $\mathbb C$-congruent to  $A_I$ and $B_I$ in $M_{I, 2k+1}$.

Thus, $f = 2(y_1y_{2k}+\cdots +y_{k}y_{k+1})$ and $g = 2(y_1y_{2k+1}+\cdots +y_{k}y_{k+2})$. The result follows from $y_i=x_i/\sqrt 2$ for $1\leq i\leq k$ and $y_{i}=x_{3k+2-i}/\sqrt 2$ for $k+1\leq i\leq 2k+1$.

{ \ \ ({\it ii}).} Let $n=2k$. For full rank $\tilde A+\lambda\tilde B$ must be strictly $\mathbb C$-congruent to $\oplus_{i=1}^r M_{III, 2m_i}$  with $\sum_{i=1}^r 2m_i =2k$, where every $\alpha_i\ne 0$ is a finite root of an elementary divisor pair $(\alpha_i+\lambda)^{2m_i}$ of $\det(M_{III, 2m_i})$. Moreover, two regular pencils are strictly congruent if and only if they have the same elementary divisors (e.g. see ch. II, § 2 in \cite{gantmacher}). The determinant of a skew-symmetric matrix is the square of its Pfaffian, so every elementary divisor of $\det (A+\lambda B)$ has even multiplicity. We will show that every root of the Pfaffian Pf$(\tilde A+\lambda \tilde B)$, and therefore also of Pf$(A+\lambda B)$, is simple.


Let $M^{\alpha_i}_{III,2}$ have $\alpha_i$ as its finite root and consider $M=\oplus_{i=1}^k M^{\alpha_i}_{III, 2}$ with all different nonzero $\alpha_i$'s. 
Then Pf($M$) has only simple roots,  a generic matrix pencil then also must have only simple roots



Since every root of Pf$(\tilde A+\lambda\tilde B)$ is simple, the elementary divisors of $\det (\tilde A+\lambda\tilde B)=\det (A+\lambda B)$ have multiplicity exactly two. By Theorem \ref{thm:Ctypes} the matrix pencil is strictly congruent to a diagonal sum of $k$ matrices of type $M_{III,2}$ with distinct $\alpha_i$'s. 

We notice that the rank of $M=\oplus_{i=1}^k M^{\alpha_i}_{III, 2}$ is $2k$ unless $\lambda$ is equal to some $\alpha_i$, in which case the rank is $2k-2$.

Thus, $f = 2(\sum_{i=1}^k \alpha_iy_{2i-1}y_{2i})$ and  $g = 2(\sum_{i=1}^k y_{2i-1}y_{2i})$. The result follows from $y_i=x_i/\sqrt 2$. 
\end{proof}

\newsection

\section{The Hilbert series of the ideal generated by two generic quadratic forms in $E_n$} \label{sec:comb}

Let $n=2k+1$.
By Theorem \ref{thm:Cstructure}, 
if follows that the question about the minimal Hilbert series for quadratic generic forms in $E_n$ is now equivalent to determining the Hilbert series of $E_n/\lideal f,g\rideal$, where 
$f = x_1 x_{k+1} + x_2 x_{k+2}+ \ldots +x_{k} x_{2k}$ and $g= x_1 x_{k+2}+ x_2 x_{k+3}+ \ldots +x_{k} x_{2k+1}$. We get an immediate result.

\begin{proposition} \label{prop:computer}
Conjecture \ref{conj:paths} is correct for odd $n \leq 19$.
\end{proposition}

\begin{proof}
See the coefficients in Table~\ref{table:odd}, obtained by calculations in Macaulay2 \cite{M2}. 
\end{proof}

At first sight, one might suspect that it is possible to determine a Gr\"obner basis for the ideal $\lideal f,g \rideal$ with respect to a suitable monomial ordering, and then determine the Hilbert series. However, we have not had any success with this approach yet. 
Instead, we turn to the combinatorial methods, mentioned in the introduction, in order to bound the Hilbert series. 
\begin{remark}
In the case when $n$ is even, 
one of the two canonical forms contains generic coefficients, so we cannot draw similar conclusions as in the odd case. However, due to results that we derive in Proposition \ref{prop:postulation} and Proposition \ref{prop:equal}, we can determine almost all of the coefficients for small odd $n$. See Table \ref{table:even}.

\end{remark}

\begin{table}[h!]
\begin{tabular}{|c|c|c|c|c|c|c|c|c|c|c|c|}
\hline
$n\backslash s$ & $0$ & $1$ & $2$ & $3$ & $4$ & $5$ & $6$ & $7$ & $8$ & $9$ & $10$ \\
\hline
$3$ 		& $1$ 	& $3$ 	& $1$ 	& $0$ 	& $0$	&$0$	 	& $0$	&$0$		&$0$		&$0$	 	&$0$		\\
$5$ 		& $1$ 	& $5$ 	& $8$ 	& $1$ 	& $0$	&$0$	 	& $0$	&$0$		&$0$		&$0$	 	&$0$		\\
$7$		& $1$	& $7$	& $19$	& $21$ 	& $1$       &$0$	 	& $0$	&$0$		&$0$		&$0$	 	&$0$		\\
$9$ 		& $1$ 	& $9$	&$34$ 	& $66$ 	& $55$	 & $1$ 	& $0$	&$0$		&$0$		&$0$	 	&$0$		\\
$11$ 	& $1$ 	& $11$	& $53$ 	& $143$ 	& $221$  	& $144$ 	& $1$	&$0$		&$0$		&$0$	 	&$0$		\\			
$13$    	& $1$	& $13$	& $76$ 	& $260$ 	& $560$ 	& $728$ 	& $377$ 	&$1$		&$0$		&$0$	 	&$0$		\\		
$15$   	& $1$ 	& $15$	& $103$  	& $425$ 	& $1156$	& $2108$	& $2380$	&$987$	&$1$		&$0$	 	&$0$		\\
$17$   	& $1$	& $17$	& $134$	& $646$	& $2109$	& $4845$ & $7752$	&$7753$	&$2584$	&$1$		&$0$		\\
$19$   	& $1$	& $19$	& $169$	& $931$	& $3535$	& $9709$	& $19551$&$28101$&$25213$&$6765$	&$1$ 	\\
\hline
 \end{tabular}

 \caption{The coefficients $a(n,s)$ for odd $n \leq 19$ agree with the values of the Hilbert function of $E_n$ modulo the ideal generated by two generic quadratic forms, see Proposition \ref{prop:computer}.}  \label{table:odd}

 \begin{tabular}{|c|c|c|c|c|c|c|c|c|c|c|c|} 
 \hline
 $n\backslash s$ & $0$ & $1$ & $2$ & $3$ & $4$ & $5$ & $6$ & $7$ & $8$ & $9$ & $10$ \\
\hline
$4$	& $\bf{1}$ & $\bf{4}$ & $\bf{4}$ & $\bf{0}$ & $\bf{0}$	&$\bf{0}$	   	& $\bf{0}$	&$\bf{0}$		&$\bf{0}$		&$\bf{0}$	 	&$\bf{0}$	 \\
$6$	 & $\bf{1}$ & $\bf{6}$	 & $\bf{13}$ &	$\bf{8}$	&$\bf{0}$	   	& $\bf{0}$	&$\bf{0}$		&$\bf{0}$		&$\bf{0}$	 	&$\bf{0}$ &$\bf{0}$	 \\
$8$ 		&  	$\bf{1}$	& $\bf{8}$	& $\bf{26}$	& $\bf{40}$	& $\bf{16}$	&$\bf{0}$	   	& $\bf{0}$	&$\bf{0}$		&$\bf{0}$		&$\bf{0}$	 	&$\bf{0}$		\\
$10$ 	&	$\bf{1}$	& $\bf{10}$ 	& $\bf{43}$	& $\bf{100}$	& $\bf{121}$	& $\bf{32}$	& $\bf{0}$	&$\bf{0}$		&$\bf{0}$		&$\bf{0}$		&$\bf{0}$ 	\\
$12$ 	&	$\bf{1}$	&$\bf{12}$	& $\bf{64}$	& $\bf{196}$	& $\bf{364}$	& $\bf{364}$	& $\bf{64}$	&$\bf{0}$		&$\bf{0}$		&$\bf{0}$		&$\bf{0}$		\\
$14$    	& 	$\bf{1}$	&$\bf{14}$	& $\bf{89}$	& $\bf{336}$	& $\bf{820}$	& $\bf{1288}$	& $1093$	&$\bf{128}$	&$\bf{0}$		&$\bf{0}$		&$\bf{0}$ 	\\
$16$   	& $\bf{1}$	&$\bf{16}$   	& $\bf{118}$ 	& $\bf{528}$   &$\bf{1581}$	& $\bf{3264}$	& $\bf{4488}$	& $3280$	&$\bf{256}$	&$\bf{0}$		&$\bf{0}$		\\
$18$		& $\bf{1}$	&$\bf{18}$	& $\bf{151}$	& $\bf{780}$	&$\bf{2755}$	& $\bf{6954}$  & $\bf{12597}$ & $\bf{15504}$ &  $9841$  &  $\bf{512}$ & $\bf{0}$ \\
$20$ 	&  $\bf{1}$ & $\bf{20}$		& $\bf{188}$	& $\bf{1100}$	&$\bf{4466}$	&$\bf{13244}$	&  $\bf{29260}$&$\bf{47652}$ &  $53296$ & $29524$ & $\bf{1024}$ \\
\hline
 \end{tabular}
 \caption{The coefficients $a(n,s)$ for even $n \leq 20$. The bold numbers means that we have equality with the Hilbert function for two generic quadratic forms in $E_n$, according to Proposition \ref{prop:postulation} and Proposition \ref{prop:equal}. Non bold numbers are upper bounds for the Hilbert function.} 
\label{table:even}
 \end{table}

We should mention  that although that it is only in the case when $n$ is odd that we have an explicit description of the generators, the conclusions that we are
about to draw will be as strong in the even case as in the odd case.

\newsubsection

\subsection{Bounds on the Hilbert series}

\subsubsection{An upper bound}
This part consists of the proof of the following result.

\begin{theorem}\label{thm:upper}
Let $E_n$ be the exterior algebra and 
$f, g$ be two  generic quadratic forms. The dimension of the $s$-th graded component of $E_n/\lideal f,g\rideal$ is at most  $a(n,s)$.
\end{theorem}

For the proof of the upper bound, we need a couple of lemmas. We will work with a certain kind of lattice paths from $(0,0)$ to $(n+2-2s,n+2)$ and a bijection between these paths and monomials of degree $s$.

A generally admissible path of size $(n,s)$ consists of $n+2$ steps, numbered  from 0,  with the $j$-th step going from $(i,j)$ to either $(i+1,j+1)$ (right move) or to  $(i-1,j+1)$ (left move). Moreover, the 0-th and the last step goes always to the right, so there will be exactly $s$ left steps.

Given such a path $p$ we define the monomial $m_p:=x_{j_1}x_{j_2}\cdots x_{j_s}$, where $j_1,\ldots,j_s$ are the indices of the left steps (see Figure ~\ref{pic:paths}). Clearly there is a one-to-one correspondence between the  set of all the paths of size $(n,s)$ and the monomials of degree $s$ in $E_n$, which we denote by $B_{n,s}$.  
{\begin{figure}[htb!]
\label{thm5}
\centering
\includegraphics[scale=1.2]{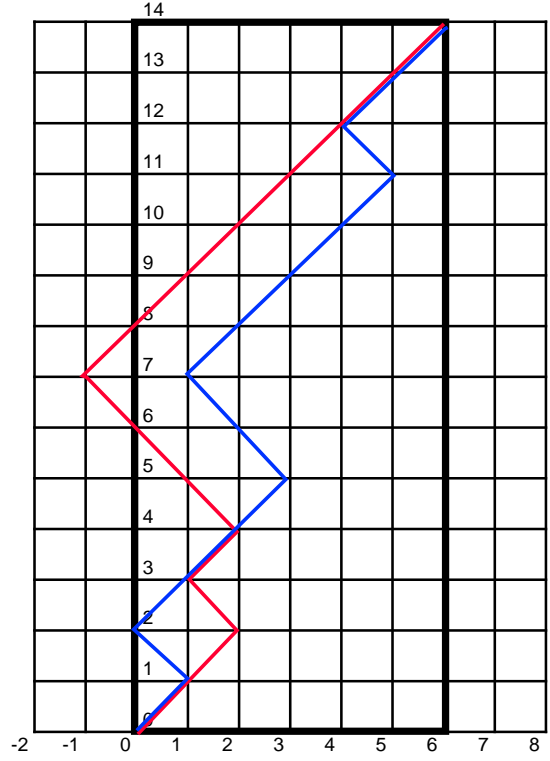}
\caption{ $(n,s)=(12,4)$. The red path corresponds to the monomial  $\color{red}{ x_2x_4x_5x_6}\in L_{12,4}$ and the blue one to the monomial $\color{blue}{ x_1x_5x_6x_{11}}\in B_{12, 4}\backslash \{ L_{12,4}\cup  R_{12,4} \}$. }
\label{pic:paths}
\end{figure}
}\

Let the subset $L_{n,s}\subset B_{n,s}$ (respectively $R_{n,s}\subset B_{n,s}$) consist of monomials of degree $s$,  whose corresponding  paths always cross the line $x=0$ (respectively   $x=n+2-2s$). Note that some  paths in $L_{n,s}$ cross the line $x=n+2-2s$ as well, which is  inevitable if, for example, $n+2\leq 2s$.

\begin{lemma}\label{lem:card}
The cardinalities of $|L_{n,s}|$ and of $|R_{n,s}|$ are given by
$$\vert L_{n,s}|=|R_{n,s}|=\begin{cases}
    0,& \text{if } s=1 \text{ or } 0  \\  
    \binom{n}{s-2},& \text{if } 2\leq 2s \leq n+2\\
    \binom{n}{s},& \text{if } 2s\geq n+2.
\end{cases}$$
\end{lemma}

\begin{proof} There is a natural bijection between the paths in $L_{n,s}$ and $R_{n,s}$, as one can see by considering the central symmetry at the point $(\frac{n+2-2s}{2},\frac{n+2}{2})$. Hence, $|L_{n,s}|=|R_{n,s}|$. In the rest of the proof we will only work with $L_{n,s}$.

The simple cases are $s=0$, where there are no left moves, and $s=1$ with exactly one left move, where the path cannot cross any vertical line. The case $n+2\le 2s$ follows easily, as any path must cross both  $x=0$ and $x=n+2-2s$, so there are no restrictions on how the left moves may occur. 

We use induction on $s$ in the remaining second case, after first fixing $n$. 
The cases $s=1$ and $2s=n+2$ (if $n$ is even) are already handled above, where $\binom{n}{-1}=0$ and $\binom{n}{\frac{n+2}{2}}=\binom{n}{\frac{n+2}{2}-2}$ respectively. Let $2<2s<n+2$ and consider the last two steps. There are only two possibilities: 
$$  {\rm either}\; \ldots \rightarrow  (n-2s,n)\rightarrow (n+1-2s,n+1) \rightarrow (n+2-2s,n+2)$$
$$ {\rm or}\; \ldots \rightarrow  (n+2-2s,n)\rightarrow (n+1-2s,n+1) \rightarrow (n+2-2s,n+2).$$
It is clear that the number of paths of the first type is $|L_{n-1, s}|$. 

To calculate the number of paths of the second type, we replace the $n$-th step that goes to the left  
with a step to the right, that is, 
$$(n+2-2s,n)\rightarrow (n+1-2(s-1),n+1).$$ It follows then that the number of paths of the second type is $|L_{n-1,s-1}|.$
Hence, we obtain the formula $|L_{n,s}|=|L_{n-1,s}|+|L_{n-1,s-1}|.$

From $2<2s<n+2$ we know $2\leq 2(s-1)<2s\leq(n-1)+2$. By the induction assumption, we get $|L_{n-1,s}|=\binom{n-1}{s-2}$ and $|L_{n-1,s-1}|=\binom{n-1}{s-3}$, and finally
 $$|L_{n,s}|=\binom{n-1}{s-2}+\binom{n-1}{s-3}=\binom{n}{s-2},$$
 where the the last equality is the well known recursive formula for binomial coefficients.
\end{proof}


\begin{lemma}\label{lem:left}
Let $f$ be a quadratic form such that, for any $s'$ and $n'\leq n$,  $$\dim\left( \lideal f \rideal_{s'}\biggr\rvert_{x_{n'+1}=\ldots=x_n=0}\right)=\min\left(\binom{n'}{s'}, \binom{n'}{s'-2}\right).$$

Let $\succ$ be the graded reverse lexicographical ordering (degrevlex) induced by 
$x_1\succ x_2\succ\ldots\succ x_n$. 
Then $L_{n,s}$ is the set of leading monomials of the $s$-th graded component of $\lideal f \rideal$.
\end{lemma}
\begin{remark} A generic quadratic form and any canonical form $$c_1x_1x_2+c_2x_3x_4+\ldots+c_{\lfloor\frac{n}{2}\rfloor}x_{2\lfloor\frac{n}{2}\rfloor-1}x_{2\lfloor\frac{n}{2}\rfloor}, \textrm{ where } 0\neq c_i\in \mathbb{C}$$ 
satisfies the condition of Lemma~\ref{lem:left}, see \cite{morenosnellman}.
\end{remark}

\begin{proof}
First we want to prove that for any monomial $m_p\in L_{n,s}$, where $p$ is the corresponding path, there is a form $r$ of degree $s-2$, such that $m_p=\lm(fr)$. 

Let $n'$ be such an index that the path $p$ contains the step $(-1,n'+1)\rightarrow (0,n'+2)$. There is at least one such $n'$, because $p$ crosses $x=0$. 
Define the part of $p$ from $(0,0)$ to $(0,n'+2)$ as  the path $p'$. Note that $n'$ is even and $s':=\frac{n'+2}{2}$ is the degree of $m_{p'}$. Now we set to zero all the variables with index ${n'+1}$ and larger: $x_{n'+1}=x_{n'+2}=\ldots=x_n=0.$

By assumption on $f$, we have 
$$\dim\left( \lideal f \rideal_{s'}\biggr\rvert_{x_{n'+1}=\ldots=x_n=0}\right)=\min\left(\binom{n'}{s'}, \binom{n'}{s'-2}\right).$$ Since $\binom{n'}{s'}= \binom{n'}{s'-2},$ this means  $\lideal f \rideal$ contains everything in degree $s'$. 
Hence, there is an $r'$ of degree $s'-2$ such that $$\lm\left(fr'\biggr\rvert_{x_{n'+1}=\ldots=x_n=0}\right)=m_{p'}.$$ Hence, with respect to degrevlex we obtain $\lm(fr')=m_{p'}.$  Then we define $r:=r'\frac{m_p}{m_{p'}}$, so that $\lm(fr)=m_{p}.$

We have proved that $L_{n,s}$ is a subset of the leading monomials. The equality follows from
$$\dim\left(\lideal f\rideal_{s}\right)=\min\left(\binom{n}{s}, \binom{n}{s-2}\right)=\begin{cases}
    \binom{n}{s},& \text{if } s\geq \frac{n+2}{2}\\
    \binom{n}{s-2},& \text{if } s\leq \frac{n+2}{2},
\end{cases}$$
as $\dim\left(\lideal f\rideal_{s}\right) = |L_{n,s}|$ by Lemma \ref{lem:card}. We conclude that $L_{n,s}$ is the set of the leading monomials of $\lideal f\rideal_{s}$.
\end{proof}

In a similar way we get the following result.

\begin{lemma}\label{lem:right}
Let $g$ be a quadratic form such that, for any $s'$ and $n'\leq n$,  $$\dim\left( \lideal g \rideal_{s'}\biggr\rvert_{x_{1}=\ldots=x_{n-n'}=0}\right)=\min\left(\binom{n'}{s'}, \binom{n'}{s'-2}\right).$$

Let $\succ$ be the degrevlex ordering induced by $x_n\succ x_{n-1}\succ\ldots\succ x_1$. 
Then $R_{n,s}$ is the set of leading monomials of the $s$-th graded component of $\lideal g\rideal$.
\end{lemma}
\begin{remark} A generic quadratic form and any canonical form $$ d_1x_nx_{n-1}+d_2x_{n-2}x_{n-3}+\ldots+d_{\lfloor\frac{n}{2}\rfloor}x_{n+2-2\lfloor\frac{n}{2}\rfloor}x_{n+1-2\lfloor\frac{n}{2}\rfloor}, \textrm{ where } 0\neq d_i\in \mathbb{C}$$
satisfies the condition of Lemma~\ref{lem:left}.
\end{remark}

\begin{proof}
Apply Lemma~\ref{lem:left} after change of variables $x_i\rightarrow x_{n+1-i}$ and $g\rightarrow f$.
\end{proof}

 We are now ready to prove the main statement.

\begin{proof}[Proof of Theorem~\ref{thm:upper}]
The key idea of the proof is  to consider two different monomial orderings on $B_{n,s}$, so that the leading monomials of $\lideal f\rideal_s$ in the first ordering and the leading monomials of $\lideal g\rideal_s$ in the second ordering correspond to the monomials whose paths intersect $x=0$ and $x=n+2-2s$ respectively.

Let $f=\sum_{1\leq i\leq j\leq n}c_{i,j} x_{i}x_{j}$ be a quadratic form satisfying the condition of Lemma~\ref{lem:left}. We choose a subset $A_s^{(f)}\subseteq B_{n,s-2}$, such that $\vert A_s^{(f)}\vert=\vert L_{n,s}\vert$ and 
$${\rm dim}\big({\rm span}\{ f\alpha\,\vert\,\ {\alpha\in A_s^{(f)}}\}\big)=\vert L_{n,s}\vert={\rm dim}( \lideal f \rideal_s).$$ Hence, 
 ${\rm \lm\big( span} \{ f\alpha\}_{\alpha\in A_s^{(f)}}\big)=L_{n,s}$ with respect to degrevlex in Lemma \ref{lem:left}.

To a vector ${\bf t}=(t_1,\ldots, t_n)\in \mathbb{C}^n$ we assign the form 
$$f^{( \bf t)}=\sum_{1\leq i\leq j\leq n}t_it_jc_{i,j} x_{i}x_{j},$$
and to a monomial $\gamma=x_{i_1}x_{i_2}\cdots x_{i_l}$ the number $t_\gamma=t_{i_1}t_{i_2}\cdots t_{i_l}.$ 

Now consider the product $f^{(\bf t)}\alpha$ for each  $\alpha\in A_s^{(f)}$, and define a matrix $M_{f^{(\bf t)}}$ of size $\vert A_s^{(f)}\vert\times \vert B_{n,s}\vert$, where the rows consist of the coefficients of $\beta\in B_{n,s}$ in $f^{(\bf t)}\alpha$. We define the matrix $M_f$ similarly for the products $f\alpha$. Given $\beta\in B_{n,s}$, we see the $(\alpha,\beta)$-entry of $M_{f^{(\bf t)}}$ is equal to $\frac{t_{\beta}}{t_{\alpha}}$ multiplied with the $(\alpha,\beta)$-entry of $M_{f}$. Hence, we have an equation between the main minors of $M_{f^{( \bf t)}}$ and $M_f$. Namely, for a subset $L'\subseteq B_{n,s}$ with $\vert L'\vert=\vert L_{s,n}\vert$, we have
$$\det( M_{f^{( \bf t)},L'}) =\det( M_{f,L'})\frac{\prod_{\beta\in L'} t_{\beta}}{\prod_{\alpha\in A_s^{(f)}} t_{\alpha}}.$$

In order to proceed we first define a sufficiently large constant $C$, that uses a constant $C_1\in \mathbb{N}$ defined in the end of this proof. Let $$C:= \left(n^{2n}+ \max_{\substack{L'\subset B_{n,s}:\\  \vert L'\vert=\vert L_{n,s}\vert}}  \frac{\vert\det(M_{f,L'})\vert}{\vert\det(M_{f,L_{n,s}})\vert} \right) C_1.$$ 

Now consider a vector ${\bf q}=(C^{-2C}, C^{-2C^2},\ldots, C^{-2C^n}).$ For monomials $\beta_1,\beta_2\in B_{n,s}$, we know that $\beta_1\prec \beta_2$ implies $C^2q_{\beta_1}\leq q_{\beta_2}$. 
Since $L_{s,n}= {\rm \lm}(fB_{n,s-2}),$ then for any subset $L_{n,s}\neq L'\subset B_{n,s}$, such that $\vert L'\vert =\vert L_{n,s}\vert$, we have either
$$\vert \det(M_{f,L'})\vert=0\;\quad {\rm  or}\;\quad \frac{\prod_{\beta\in L_{n,s}} t_{\beta}}{\prod_{\beta\in L'} t_{\beta}}\geq C^2.$$

The above together with $\vert\det(M_{f,L'})\vert<C\vert\det(M_{f,L_{n,s}})\vert$ yields the inequality 
\begin{equation}
\label{eq:det}
\vert \det(M_{f^{( \bf t)}, L_{n,s}})\vert > C \vert \det(M_{{f^{( \bf t)}},L'})\vert.
\end{equation}

Let $g$ be a quadratic  form satisfying the condition of Lemma~\ref{lem:right}. Similarly as above, there is a subset $A_s^{(g)}\subset B_{n,s-2}$, so that, with respect to degrevlex in Lemma~3,  Lm(span$\{g\alpha\,\vert\,\alpha\in A_s^{(g)}\})=R_{n,s}$. The matrix $M_g$ is defined by $g$ in the corresponding way as $M_f$ by $f$. 

We know that the main minor $\det(M_{g,R_{n,s}})$ is non-zero. Thus there is subset  $\bar{A}_s^{(g)}\subseteq A_s^{(g)}$ of size $|R_{n,s}\setminus L_{n,s}|$ such that the determinant of the minor $\bar{A}_s^{(g)}\times (R_{n,s}\setminus L_{n,s})$ does not vanish. Further, let $M_g'$ be a submatrix of $M_g$, corresponding to $\bar{A}_s^{(g)}\times B_{n,s}$. 

We construct a big matrix $\mathcal{M}$ of size $(|A_s^{(f)}|+|\bar{A}_s^{(g)}|) \times |B_{n,s}|$, where the first $|A_s^{(f)}|$ rows come from $M_{f^{({\bf q})}}$ and the last $|\bar{A}_s^{(g)}|$ rows are from $M_g'$. Since $|A_s^{(f)}|+|\bar{A}_s^{(g)}|=|L_{n,s}\cup R_{n,s}|$, it is enough to show that the main minor $\det(\mathcal{M}_{L_{n,s}\cup R_{n,s}})$ does not vanish.
We can rewrite the minor as
$$\det(\mathcal{M}_{L_{n,s}\cup R_{n,s}})=\sum_{\substack{L'\subset L_{n,s}\cup R_{n,s}:\\  |L'|=|L_{n,s}|}} \pm \det(M_{f^{(q)},L'})\cdot\det({M_{g, \{ L_{n,s}\cup R_{n,s}\}\setminus L'}}).$$
Due to Equation~\eqref{eq:det}, it is possible to choose a sufficiently large $C_1$, such that 
\begin{multline}
|\det(M_{f^{({\bf q})},L_{n,s}})|\cdot |\det({M_{g,  R_{n,s}\setminus L_{n,s}}})|>\\ 
\left(\frac{|\det(M_{f^{({\bf q})},L_{n,s}})|}{C}\right)
\left(\sum_{\substack{L_{n,s}\neq L'\subset L_{n,s}\cup R_{n,s}:\\  |L'|=|L_{n,s}|}}   |\det({M_{g, \{ L_{n,s}\cup R_{n,s}\}\setminus L'}})|\right)>\\
\sum_{\substack{L_{n,s}\neq L'\subset L_{n,s}\cup R_{n,s}:\\  |L'|=|L_{n,s}|}}  |\det(M_{f^{({\bf q})}L'})|\cdot |\det({M_{g, \{ L_{n,s}\cup R_{n,s}\}\setminus L'}})|.
\end{multline}

Altogether, for a large $C_1$ the main minor of $\mathcal{M}_{L_{n,s}\cup R_{n,s}}$ will not vanish, which means that the dimension of the $s$-graded component of the ideal generated by $\lideal f^{({\bf t})}, g\rideal$ is at least $|L_{n,s}\cup R_{n,s}|$. Thus, the dimension of the quotient algebra is at most $a(n,s)$. This concludes the proof.
\end{proof}

\subsubsection{A lower bound}

We recall the following result, which is the exterior version of the lower bound used by Fr\"oberg in the commutative setting.  
\begin{lemma}
Let $f$ and $g$ be generic quadratic forms in $E_n$. Then $$\HS(E_n/\lideal f,g\rideal,t) \geq \left[(1+t)^n(1-t^2){^2}\right].$$
\end{lemma}

\begin{proof}  
From \cite{morenosnellman}, we know that $\HS(E_n/\lideal f\rideal,t) = [(1+t)^n(1-t^2)].$ The multiplication map $\cdot g$ from $(E_n/\lideal f\rideal)_i$ to $(E_n/\lideal f\rideal)_{i+2}$ has rank at most 
$$\min\big(\dim (E_n/\lideal f\rideal)_i{}, \dim (E_n/\lideal f\rideal)_{i+2}{}\big).$$
Thus, $$\dim \left(E_n/\lideal f,g\rideal \right)_{i+2} \geq \max\big(0, {\dim} (E_n/\lideal f\rideal)_{i+2} - {\dim}(E_n/\lideal f\rideal)_{i}{}\big),$$ that is, $$\HS(E_n/\lideal f,g\rideal,t) \geq \left[(1+t)^n(1-t^2){^2}\right].$$
\end{proof} 
\newsubsection

\subsection{Coefficients in the Hilbert series}
Using the bounds that we have derived in the previous section, we can determine a majority of the coefficients in the Hilbert series. 

\begin{proposition} \label{prop:postulation}

Let $f$ and $g$ be two generic quadratic forms in $E_n$.

\begin{itemize}
\item 
If $n$ is odd and equal to $2k+1$, then 
$$\dim((E_{2k+1}/\lideal f,g\rideal)_{k+1}) = a(2k+1,k+1) = 1,$$

and 
$$\dim((E_{2k+1}/\lideal f,g\rideal)_{k+2}) = a(2k+1,k+2) = 0.$$

\item
If $n$ is even and equal to $2k$, then
$$\dim((E_{2k}/\lideal f,g\rideal)_k) = a(2k,k) = 2^k,$$
and 
$$\dim((E_{2k}/\lideal f,g\rideal)_{k+1}) = a(2k,k+1) = 0.$$
\end{itemize}

\end{proposition}

\begin{proof}
From the definition of the numbers $a(n,s)$, it follows immediately that $a(2k+1,k+2) = a(2k,k+1) = 0$, proving the second statement in each case. 

We now concentrate on the first statements. We begin with the odd case.
By Theorem  \ref{thm:Cstructure}, we can assume that $f=x_1 x_{2k} + x_2 x_{2k-1} + \cdots + x_{k} x_{k+1} $ and 
$g = x_1 x_{2k+1} + x_2 x_{2k} + \cdots + x_k x_{k+2}$.  Now it is easy to see that $x_{k+1} \cdots x_{2k+1} \notin \lideal f,g\rideal$, so $\dim((E_{2k+1}/\lideal f,g\rideal)_{k+1}) \geq1.$ On the other hand, we know by Theorem \ref{thm:upper} that  $\dim((E_{2k+1}/\lideal f,g\rideal)_{k+1}) \leq a(2k+1,k+1)$. But there is only one path inside the rectangle $1 \times (2k+3)$, and $a(2k+1,k+1) = 1$.

We proceed in the same manner in the even case. By Theorem \ref{thm:Cstructure} we can assume that $f=x_1 x_{k+1} + x_2 x_{k+2} + \cdots + x_k x_{2k} $ and $g=\alpha_1 x_1 x_{k+1} + \alpha_2 x_2 x_{k+2} + \cdots + \alpha_k x_k x_{2k}$. Now we observe that the monomial $y_1 y_2 \cdots y_k$ with $y_i \in \{x_{2i-1},x_{2i}\}$ lies outside $\lideal f,g\rideal$, so $\dim((E_{2k}/\lideal f,g\rideal)_k) \geq 2^k.$ By Theorem \ref{thm:upper} it is enough to show that $a(2k,k) = 2^k$. The possible paths inside the rectangle $2 \times (2k+2)$ are $(0,0) \to (1,1) \to (0/2,2) \to (1,3) \to (0/2,4)\rightarrow \cdots \rightarrow (0/2,2k) \to (1,2k+1) \to (2,2k+2)$, hence, the requested number is indeed $2^k$.\end{proof}

\medskip

\begin{proposition} 
\label{prop:equal}
Let $f,g$ be generic quadratic forms  in $E_n$. Let  $s$ and $n$ be integers such that $s\leq \lfloor \frac{n}{3} \rfloor +1$.  
Then the dimension of the $s$-th graded component of $E_n/\lideal f,g\rideal$ is equal to  $a(n,s)=\binom{n}{s} - 2\binom{n}{s-2} + \binom{n}{s-4}$.
\end{proposition}
\begin{proof}
Since $[(1+t)^n(1-t^2)^2] = (1+t)^n(1-2t^2+t^4)]$ is a lower bound, and $a(n,s)$ is an upper bound, it is enough to show that 
$a(n,s)=  \binom{n}{s} - 2\binom{n}{s-2} + \binom{n}{s-4}$ when $s \leq  \lfloor \frac{n}{3} \rfloor +1$.






By Lemma~\ref{lem:card} the number of paths crossing the $y$-axis or the line  $x=n+2-2s$ is the same, $n\choose{s-2}$. Hence, the inclusion-exclusion principle yields $a(n,s) = {n\choose s} - {n\choose s-2}- {n\choose s-2} +r$, where $r$ is the number of paths that  cross both $x=0$ and $x=n+2-2s$ before ending at $(n+2-2s, n+2)$.

If a path $p$ from $(1,1)$ to $(n+1-2s,n+1)$ crosses the line  $x=n+2-2s$ before the line $x=0$, then the length of  $p$ is at least $(n+2-2s)+(n+2-2s+2)+(n+2-2s)=3n+8-6s$. Hence,
$$3n+8-6s\leq \textrm{length}(p)=n$$
$$2n+2 + 6 \leq 6s$$
$$ \frac{n+1}{3}  + 1 \leq  s,$$
but we have that $s\leq  \frac{n}{3}+1$, so there is no such path.

%

We get that all paths (that cross both $x=0$ and $x=n+2-2s$) should cross the line $x=0$ before the line $x=n+2-2s$. For any such path $p$ from $(1,1)$ to $(n+1-2s,n+1)$, we chose the first point $(-1,a_p)\in p$ and the first point $(n+3-2s,b_p)\in p$. Reflect this path from the step $a_p$ until the step $b_p$, i.e., change the dirrection of all steps $a_p\leq i < b_p$ (i.e., we change  steps of type $(*,*)\to (*+1,*+1)$ to steps $(*,*)\to (*-1,*+1)$ and the converse). We get the new path $p'$, which is from the point $(1,1)$ to the point $(n+1-2s- 2(n+4-2s) ,n+1)=(2(s-4)-n+1,n+1)$. 
It is possible to reverse this procedure, than the number of such $p$ is exactly the number of $p'$ (without restrictions). 
Then in the case $s\leq \lfloor\frac{n}{3} \rfloor+1$, the number of paths $r$ is equal to $\binom{n}{s-4}$.
\end{proof}

\newsection
 
 \section{Squares of generic linear forms in the square free algebra} \label{sec:squarefree}

We will now turn to commutative setting. Recall that there is a conjecture by Iarrobino \cite{Iarrobino} stating that 
$\mathbb{C}[x_1,\ldots,x_n]/\lideal l_1^{d}, \ldots, l_r^{d}\rideal$  and  $\mathbb{C}[x_1,\ldots,x_n]/\lideal f_1,\ldots,f_{r}\rideal$ have the same Hilbert series, where 
$l_1,\ldots, l_r$ are generic linear forms and the $f_i$ are generic forms of degree $d$, except for the cases $r = n+2, r = n+3, (n,r) = (3,7), (3,8), (4,9), (5,14)$. We will now focus on the Hilbert series for the simplest exceptional case; $r = n+2$ and $d = 2$.
After a linear change of variables, it is enough to consider $S_n/\lideal x_1^2,\ldots, x_n^2,\ell_1^2,\ell_2^2\rideal$, where $l_1$ and $l_2$ are generic linear forms. 


\subsection{Bounds on the Hilbert series}
\begin{theorem}\label{thm:upper:sq-free}
Let $S_n$ be the polynomial ring in $n$ variables and let
$\ell_1, \ell_2$ be two  generic linear forms. 
Then the dimension of the $s$-th graded component of the quotient $S_n/\lideal x_1^2,\ldots, x_n^2,\ell_1^2,\ell_2^2\rideal$ is at most  $a(n,s)$.
\end{theorem}
\begin{proof}
We can use the same proof as in Theorem ~\ref{thm:upper}. Indeed, the references to Lemma \ref{lem:left} and  Lemma \ref{lem:right} are valid, since for a generic linear $\ell$, the Hilbert series of $S_n/\lideal x_1^2,\ldots, x_n^2,\ell^2 \rideal$ is equal to $[(1+t)^n(1-t^2)]$ (see~\cite{stanley}). 
\end{proof}
\newsubsection

\subsection{Proof of Theorem \ref{thm:conj12}}

Let $a_e(n,s)$ and $a_c(n,s)$ be the dimensions of the $s$-graded components of $E_n/\lideal f,g\rideal$ and of $S_n/\lideal x_1^2,\ldots, x_n^2,\ell_1^2,\ell_2^2\rideal$ resp., where $f,g\in E_n$ are generic quadratic forms and $\ell_1,\ell_2\in S_n$ are generic linear forms.
\begin{remark}
In fact, the algebra $S_n/\lideal x_1^2,\ldots, x_n^2,\ell_1^2,\ell_2^2\rideal$ is isomorphic to 
 the algebra 
and $S_{n+2}/\lideal x_1^2,\ldots,x_{n+1}^2,x_{n+2}^2,\ell_1^*,\ell_2^*\rideal,$ where $\ell_1^*=\ell_1-x_{n+1}$ and $\ell_2^*=\ell_2-x_{n+2}$. 

Thus, we can define the coefficients $a_c(n,s)$ for $n\geq-2$ and $s\geq 0$. 
\end{remark}

\begin{lemma}\label{lem:rec-conj}

Let $k$ and $s$ be positive integers. Then
$$a_e(2k,s)=\sum_{\substack{r\in\mathbb{N}:\\ 2|(s-r)}} 2^r\binom{k}{r}a_c\left(k-r-2,\frac{s-r}{2}\right).$$
\end{lemma}
\begin{proof}

By Theorem \ref{thm:Cstructure}, we can assume that 
$$f=b_1x_1x_2+b_2x_3x_4+\ldots+b_kx_{2k-1}x_{2k},$$
and
$$g=c_1x_1x_2+c_2x_3x_4+\ldots+c_kx_{2k-1}x_{2k}.$$

For a subset $I \subset [k] := \{1,2,\ldots,k\},$ let 

$$f_I = \sum_{i \in [k] \setminus I} b_i x_{2i-1} x_{2i}$$
and
$$g_I = \sum_{i \in [k] \setminus I} c_i x_{2i-1} x_{2i}.$$

Denote by $E_I$ the subalgebra of $E_n$ generated by $\{x_{2i-1} x_{2i}\}, i \in [k] \setminus I]$. Notice that $f_I,g_I \subseteq E_I$, so we can form
$R_I := E_I/(f_I,g_I)$.

For a subset $J\subseteq I$ we define the monomial
$$m(I,J)=\prod_{j\in J} x_{2j-1} \prod_{i\in I\setminus J}x_{2i} \in E_n.$$

Since $x_{2i-1} \cdot x_{2i-1} x_{2i} = x_{2i} \cdot x_{2i-1} x_{2i} = 0$ in $E_n$, if follows that 
$$\{m(I,J) \cdot e, J \subseteq I \subseteq [k] \text{ and } e \text{ a basis element of } R_I\}$$
is a basis for $E_n$ mod $\lideal f,g \rideal$, where we in the product $m(I,J) \cdot e$ regard $e$ as an element in $E_n$.

Now
$$a_e(2k,s)=\dim(E_{2k}/{\lideal f,g\rideal})_s) = \sum_{\substack{I \subseteq [k] \\ 2 | s - |I|}} 2^{|I|} \dim (R_I)_{s-|I|},$$
since we have $2^{|I|}$ possibilities to choose $J$.

But $R_I$ is a commutative algebra isomorphic to $$\mathbb{C}[y_i, i \in [k] \setminus I]/\lideal \sum_{i \in [k] \setminus I} b_i y_i, \sum_{i \in [k] \setminus I} c_i y_{i} \rideal,$$ 
each $y_i$ of degree two, so
\begin{align*}
 \sum_{\substack{I \subseteq [k] \\ 2 | s - |I|}} 2^{|I|} \dim (R_I)_{s-|I|} &=  \sum_{\substack{I \subseteq [k] \\ 2 | s - |I|}} 2^{|I|} a_c(k-|I| - 2,\frac{s- |I|}{2}) \\
 &=\sum_{\substack{r \in \mathbb{N} \\ 2 | s - r}} 2^{r} \binom{k}{r} a_c(k-r - 2,\frac{s- r}{2}).
\end{align*}

\end{proof}

\begin{lemma}\label{lem:rec-path}
Let $k$ and $s$ be positive integers. Then
$$a(2k,s)=\sum_{\substack{r\in\mathbb{N}:\\ 2|(s-r)}} 2^r\binom{k}{r}a\left(k-r-2,\frac{s-r}{2}\right).$$
\end{lemma}
\begin{proof}
For a path $p$ we consider a pair of steps such that
$$(2x+1,y)\to(2x+1\pm 1,y+1)\to (2x+1,y+2),$$
i.e., one step to right or left and another to the opposite. 
Note that if $p$ is inside rectangle $(2k+2-2s)\times (2k+2)$, then after deleting these two steps the new path $p'$ is inside $(2k+2-2s)\times 2k$. 
Furthermore, for any path $q$ inside $(2k+2-2s)\times 2k$, if we insert such two steps in any odd place then the new paths $q'$ (left-right) and $q''$ (right-left) are both inside $(2k+2-2s)\times (2k+2)$.

For any path $p$ inside the rectangle $(2k+2-2s)\times (2k+2)$, we collect the $y$-coordinate for the left-right steps in a set $I_1\subseteq [k]$ and the $y$-coordinate for the right-left steps in a set  $I_2\subseteq [k]$. 
After deleting all such pairs we get a path inside $(2k+2-2s)\times (2k+2-2|I_1|-2|I_2|)$. We obtain
$$a(2k,s)=\sum_{I_1\sqcup I_2\subseteq [k]} b(2k+2-2s,2k+2-2|I_1|-2|I_2|),$$

where $b(2k+2-2s,2k+2-2r)$ is the number of paths inside $(2k+2-2s)\times (2k+2-2r)$ such that the first step is to the left, the last to right, and all other steps go by pairs; both to the right or both to the left. Note that if $(s-r)$ is odd, then $b(2k+2-2s,2k+2-2r)=0$. Indeed, let $u$ be the number of ``left'' pairs, then 
$$2k+2-2s=2k+2-2r-4u\ \ \Longleftrightarrow \ \ r-s=2u.$$
 
It follows that 
$$\sum_{I_1\sqcup I_2\subseteq [k]} b(2k+2-2s,2k+2-2|I_1|-2|I_2|)=\sum_{\substack{r\in\mathbb{N}:\\ 2|(s-r)}} 2^r\binom{k}{r}b(2k+2-2s,2k+2-2r).$$
Now  $b(2k+2-2s,2k+2-2r)$ is equal to the number of paths inside $(2k-2s)\times (2k-2r)$ where all steps are doubled, i.e., from $(x,y)$ to $(x-2,y+2)$ or $(x+2,y+2)$.  Hence, $b(2k+2-2s,2k+2-2r)$ is equal to the number of paths inside $(k-s)\times (k-r)$, i.e., 
$$b(2k+2-2s,2k+2-2r)=a(k-s-2, \frac{s-r}{2}),$$
concluding the proof.

\end{proof}

\begin{proof}[Proof of Theorem~\ref{thm:conj12}]
If Conjecture~\ref{conj:paths-sqfree} holds for any number of variables, then from Lemma~\ref{lem:rec-conj} and Lemma~\ref{lem:rec-path} it follows that Conjecture~\ref{conj:paths} should hold for any even number of variables.

Assume that Conjecture~\ref{conj:paths} holds for any even number of variables. Then 
$$a(2k,s) = \sum_{\substack{r\in\mathbb{N}:\\ 2|(s-r)}} 2^s\binom{k}{r}a\left(k-r-2,\frac{s-r}{2}\right)$$
by Lemma \ref{lem:rec-conj}.

On the other hand, by Lemma \ref{lem:rec-path}, it holds that
$$a_e(2k,s)=\sum_{\substack{r\in\mathbb{N}:\\ 2|(s-r)}} 2^s\binom{k}{r}a_c\left(k-r-2,\frac{s-r}{2}\right),$$
so the assumption $a(2k,s) = a_e(2k,s)$ gives us
$$\sum_{\substack{r\in\mathbb{N}:\\ 2|(s-r)}} 2^s\binom{k}{r}a\left(k-r-2,\frac{s-r}{2}\right)=\sum_{\substack{r\in\mathbb{N}:\\ 2|(s-r)}} 2^s\binom{k}{r}a_c\left(k-r-2,\frac{s-r}{2}\right).$$
By Theorem~\ref{thm:upper:sq-free} we know that 
$$a(n,m)\geq a_c(n,m),$$
so we get the equality
$$a\left(k-r-2,\frac{s-r}{2}\right)=a_c\left(k-r-2,\frac{s-r}{2}\right),$$
for any $r\leq s$ such that $2|(s-r)$. Hence, for any pair $m\leq n$ we have $a(n,m)= a_c(n,m).$
\end{proof}

It is natural to ask if there might be other connections between the series of quotients of $E_n$ and quotients of $S_n$. 
We end up by giving a question in this direction, whose formulation is based upon computer experiments.

\begin{question} \label{question:gen} Let $\ell_1, \ell_2 \in S_n$ be two generic linear forms, and let $f, g \in E_n$ be two generic quadratic forms. Do the two algebras 
$$S_n/(x_1^2,\ldots,x_n^2,\ell_1^{2 d_1}, \ell_2^{2 d_2})$$ and 
$$E_n/(f^{d_1},g^{d_2})$$
have the same Hilbert series?
\end{question}
\begin{remark}
In the case $d_1=d_2 = 1$, Question \ref{question:gen} is a weaker version of Conjecture \ref{conj:paths-sqfree}.
\end{remark}
 
 \medskip\noindent
{\bf Acknowledgements.} The authors want to thank the participants of the Stockholm problem solving seminar for many fruitful discussions during the progress of this work.

%
%

\end{document}